\documentclass[12pt, notitlepage]{amsart}
\usepackage{latexsym, amsfonts, amsmath, amssymb, amsthm, cite, tabls, paralist}

\usepackage[final]{showlabels}

\newtheorem{theorem}{Theorem}[section]
\newtheorem{lemma}[theorem]{Lemma}
\newtheorem{proposition}[theorem]{Proposition}
\newtheorem*{nonumlemma}{Lemma}
\newtheorem*{nonumtheorem}{Theorem}
\newtheorem{corollary}[theorem]{Corollary}
\newtheorem{conjecture}[theorem]{Conjecture}

\theoremstyle{remark}
\newtheorem{remark}[theorem]{Remark}

\theoremstyle{definition}

\theoremstyle{remark}

\numberwithin{equation}{section}

\newcommand{\Z}{\mathbb{Z}}
\newcommand{\Q}{\mathbb{Q}}
\newcommand{\R}{\mathbb{R}}

\newcommand{\C}{\mathbb{C}}

\usepackage{graphicx}
\usepackage{mathrsfs}

\usepackage{varioref}

\newcommand{\unif}{\text{unif}}

\begin{document}

\title{On an inverse ternary Goldbach problem}
\author{Xuancheng Shao}
\address{Department of Mathematics \\ Stanford University \\
450 Serra Mall, Bldg. 380\\ Stanford, CA 94305-2125}
\email{xshao@math.stanford.edu}

\maketitle

\begin{abstract}
We prove an inverse ternary Goldbach-type result. Let $N$ be sufficiently large and $c>0$ be sufficiently small. If $A_1,A_2,A_3\subset [N]$ are subsets with $|A_1|,|A_2|,|A_3|\geq N^{1/3-c}$, then $A_1+A_2+A_3$ contains a composite number. This improves on the bound $N^{1/3+o(1)}$ obtained in \cite{PSS88} using Gallagher's larger sieve. The main ingredients in our argument include a type of inverse sieve result in the larger sieve regime, and a variant of the analytic large sieve inequality.
\end{abstract}

\section{Introduction}

An old conjecture of Ostmann, sometimes called the inverse Goldbach problem, says that there are no nontrivial additive decompositions of the set of primes. In other words, there do not exist subsets $A_1,A_2\subset\Z$ with $|A_1|,|A_2|>1$ such that for all sufficiently large $n\in\Z$, $n\in A_1+A_2$ if and only if $n$ is prime. Here $A_1+A_2=\{a_1+a_2:a_1\in A_1,a_2\in A_2\}$. Its ternary analogue is solved by Elsholtz \cite{Els01}.

\begin{theorem}
There do not exist subsets $A_1,A_2,A_3\subset\Z$ with $|A_1|,|A_2|,|A_3|>1$ such that for all sufficiently large $n\in\Z$, $n\in A_1+A_2+A_3$ if and only if $n$ is prime.
\end{theorem}

For more references on this problem see \cite{Els06,EH13} and the survey \cite{Els09}. In this paper we study additive decompositions of \textit{subsets} of primes. More precisely, can we find large sets $A_1,A_2\subset\Z$ such that if $n\in A_1+A_2$ then $n$ is prime? Heuristically, the answer to this question should be no (for an appropriate meaning of largeness) because the primes behave randomly from an additive point of view. This problem was discussed in \cite{PSS88,Els06,GH14}.

\begin{conjecture}\label{conj:gh}
Let $\delta>0$. The following holds for $N$ sufficiently large depending on $\delta$. If $A_1,A_2\subset [N]$ are subsets with $|A_1|,|A_2|\geq N^{\delta}$, then $A_1+A_2$ contains a composite number.
\end{conjecture}

Here $[N]=\{1,2,\cdots,N\}$.
The conjecture is open for $\delta\leq 1/2$. Indeed, if this is true for any $\delta<1/2$ then the inverse (binary) Goldbach problem follows. In the other direction, there exist $A_1,A_2\subset [N]$ with $|A_1|,|A_2|\geq \log N/\log\log N$ such that $A_1+A_2$ is contained in the primes (see Corollary 1.3.6 in \cite{Els09}). Note also the similarity between Conjecture \ref{conj:gh} and the problem of finding the clique numbers of Paley sum graphs, constructed using quadratic residues in a finite field.

Our main result provides a nontrivial bound for the ternary analogue of Conjecture \ref{conj:gh}.

\begin{theorem}\label{thm:main}
Let $N$ be sufficiently large and $c>0$ be sufficiently small. If $A_1,A_2,A_3\subset [N]$ are subsets with $|A_1|,|A_2|,|A_3|\geq N^{1/3-c}$, then $A_1+A_2+A_3$ contains a composite number.
\end{theorem}

If $N^{1/3-c}$ above is replaced by $N^{1/3+o(1)}$, then the result follows from Gallagher's larger sieve; see Theorem 3 in \cite{PSS88}. There are two main ideas in the proof of Theorem \ref{thm:main}, which we discuss below. The first idea is to get an improved bound in sieving situations when there are additional additive structures available. The second idea is to develop a variant of the analytic large sieve inequality that is better suited in certain circumstances.

\subsection{Sieving with additive structures}

To see the connection between Theorem \ref{thm:main} and sieving problems, if we assume that all elements of $A_1+A_2+A_3$ are prime, then after some pruning process we may conclude that $A_1+A_2+A_3$ misses the residue class $0\pmod p$ for $p$ up to some threshold. This in turn implies that the individual sets $A_1,A_2,A_3$ can occupy at most about $p/3$ residue classes on average. By Gallagher's larger sieve, this leads to the bound $N^{1/3+o(1)}$ instead of $N^{1/3-c}$; more details can be found in the proof of Proposition \ref{prop:main-weak}.

In fact, we can say more than $A_1,A_2,A_3$ occuping at most $p/3$ residue classes. Using a Freiman-type result in additive combinatorics (Lemma \ref{lem:gryn} below), we may furthermore assume that $A_1,A_2,A_3\pmod p$ are contained in an \textit{arithmetic progression} of length slightly above $p/3$. Under this further assumption we can indeed improve the larger sieve.

\begin{theorem}\label{thm:sieve-ap-simple}
Let $A\subset [N]$ be a subset. Let $\alpha\in [1/3,1/2]$ be real. Let $c>0$ be sufficiently small. If, for each prime $p\leq N^{\alpha}$, the residues $A\pmod p$ lie in an arithmetic progression $S_p\subset\Z/p\Z$ of length $\alpha p$, then $|A|\ll N^{\alpha-c}$.
\end{theorem}

When $\alpha=1/2$, the set of squares up to $N$ has size about $N^{1/2}$, and occupies about $p/2$ residue classes modulo each prime $p$. Of course these residue classes are exactly the quadratic residues, and should certainly be far away from being an arithmetic progression. The inverse large sieve conjecture roughly says that either the set $A$ possesses some algebraic structure like being the squares, or the size of $A$ is much smaller than predicted by the large sieve. For more evidences on the inverse large sieve conjecture see \cite{HV09,Wal12,GH14}.

We expect Theorem \ref{thm:sieve-ap-simple} to hold for any $\alpha>0$ (with some $c>0$ sufficiently small depending on $\alpha$), and this will imply the inverse Goldbach-type result with more than three summands. Unfortunately our argument is not sufficient for this purpose.

\subsection{A variant of the large sieve}

Recall the traditional large sieve inequality. For a compactly supported function $f:\Z\rightarrow\C$, its Fourier transform is defined by
\[ \hat{f}(x)=\sum_n f(n)e(-xn) \]
for $x\in\R/\Z$, where $e(xn)=e^{2\pi i xn}$.

\begin{theorem}[Analytic large sieve]\label{thm:large-old}
Let $f:[N]\rightarrow\C$ be an arbitrary function. Let $x_1,\cdots,x_m\in\R/\Z$ be points that are at least $\delta$-spaced, meaning that $\|x_i-x_j\|\geq\delta$ for all $i\neq j$, where $\|x\|$ is the distance from $x$ to its nearest integer. Then
\[ \sum_{i=1}^m |\hat{f}(x_i)|^2\ll (N+\delta^{-1})\sum_{n=1}^N |f(n)|^2. \]
\end{theorem} 

For proofs and applications of it see \cite{Mon78,Bom87}. In sieve problems, the large sieve inequality is usually applied with the points $\{x_1,\cdots,x_m\}$ being the set of reduced fractions with denominator at most some parameter $P$ that is usually free to choose. This set of points is $\delta$-spaced with $\delta=P^{-2}$. Because of the term $N+P^2$ that appears in the upper bound, the parameter $P$ is best taken to be about $N^{1/2}$.

It turns out that, in our application, the parameter $P$ has to be much smaller than $N^{1/2}$, and the traditional large sieve does not give a satisfactory bound. The following variant serves as a substitute.  

\begin{theorem}[Large sieve variant]\label{thm:large}
Let $f:[N]\rightarrow\C$ be an arbitrary function. For any prime $p$ let
\[ I_p(f)=\sum_{a=1}^{p-1} |\hat{f}(a/p)|^2. \]
For any positive integer $k$ we have
\[ \sum_{\substack{p\leq P\\ p\text{ prime}}}I_p(f)^{k/(2k-1)}\ll_k P^{(k-1)/(2k-1)}(N^{1/(2k-1)}+P^{2k/(2k-1)})\sum_{n=1}^N |f(n)|^{2k/(2k-1)}.  \]
\end{theorem}

The $k=1$ case is exactly Theorem \ref{thm:large-old} with $\{x_1,\cdots,x_m\}$ the set of reduced fractions $a/p$ with $p\leq P$ . The first term on the right $P^{(k-1)/(2k-1)}N^{1/(2k-1)}$ is necessary, as can be seen by taking $f$ to be the characteristic function of the set of multiples of some fixed prime $p_0\sim P$. Therefore, the inequality is sharp (apart from the implied constant) when $P\leq N^{1/2k}$.

As an immediate corollary, we have:

\begin{corollary}\label{cor:vaughan}
Let $A\subset [N]$ be a subset. Let $k$ be a positive integer. If, for each prime $p\leq N^{1/2k}$, the set $A$ misses at least $0.1p$ residue classes modulo $p$, then $|A|\ll_k N^{1/2}$.
\end{corollary}

\begin{proof}
Apply Theorem \ref{thm:large} with $f=1_A$, $P=N^{1/2k}$, and use the fact that $I_p(f)\gg |A|^2$.
\end{proof}

This corollary is not new; indeed it follows from Theorem 3.1 in \cite{Vau73} together with the usual large sieve. However, the method used in \cite{Vau73} only applies to functions $f$ supported on a well-sieved set. In our application, Theorem \ref{thm:large} will be applied to a more general function.

\subsection{Outline of the paper}

The rest of the paper is organized as follows. In Section \ref{sec:prelim} we state and prove some preliminary lemmas. In Section \ref{sec:small-doubling}, we prove an improved larger sieve result assuming additive structures, in the spirit of the inverse sieve conjecture. In Section \ref{sec:large}, we prove Theorem \ref{thm:large}, our variant of the analytic large sieve. In Section \ref{sec:sieve-ap}, we deduce Theorem \ref{thm:sieve-ap-simple}, which is the foundation in our proof of Theorem \ref{thm:main}, given in Section \ref{sec:proof-main}.

We adopt the convention that, whenever parameters such as $\alpha,c,c',\epsilon,k$ occur in a statement, the implied constant can always depend on these parameters and the positive integer $N$ is always assumed to be sufficiently large. The letter $p$ is reserved to denote a prime number.

\smallskip

\textbf{Acknowlegement.} The author thanks Ben Green, Adam Harper, Dimitris Koukoulopoulos, and Kannan Soundararajan for helpful discussions. He is also grateful to Christian Elsholtz for help with references.

\section{Preliminary Lemmas}\label{sec:prelim}

We begin with the basic estimate that will be used frequently:
\[ \sum_{p\leq Q}\frac{\log p}{p}=\log Q+O(1). \]

\subsection{Gallagher's larger sieve}

Gallagher's larger sieve roughly says that, if a subset $A\subset [N]$ occupies at most $\alpha p$ residue classes for each prime $p\leq N^{\alpha}$, then $|A|\ll N^{\alpha}$. Here is a more precise version, whose proof can be found in \cite{Gal71}.

\begin{lemma}[Gallagher's larger sieve]\label{lem:gallagher}
Let $A\subset [N]$ be a subset and $\alpha\in (0,1)$ be real. If
\[ \sum_{p\leq N^{\alpha}}\frac{\log p}{|A\pmod p|}>\log N+1, \]
then $|A|\ll N^{\alpha}$.
\end{lemma}

We will often be in a situation where $|A\pmod p|$ appears in the numerator. The following simple lemma will then be useful.

\begin{lemma}\label{lem:am-gm}
Let $I$ be a finite index set. For each $i\in I$, let $w_i,a_i>0$ be reals. If
\[ \sum_{i\in I}w_ia_i=\kappa\sum_{i\in I}w_i \]
for some $\kappa>0$, then
\[ \sum_{i\in I}\frac{w_i}{a_i}\geq\frac{1}{\kappa}\sum_{i\in I}w_i. \]
\end{lemma}

\begin{proof}
For each $i\in I$, we have $1/a_i+a_i/\kappa^2\geq 2/\kappa$. The result follows from multiplying this inequality by $w_i$ and summing over $i$.
\end{proof}

Combining the previous two lemmas, we quickly deduce:

\begin{lemma}\label{lem:gallagher2}
Let $A\subset [N]$ be a subset and $\alpha\in (0,1)$ be real. If
\[ \sum_{p\leq N^{\alpha}}\frac{\log p}{p}\cdot\frac{|A\pmod p|}{p}<(\alpha-c)\sum_{p\leq N^{\alpha}}\frac{\log p}{p} \]
for some $c>0$, then $|A|\ll N^{\alpha}$.
\end{lemma}

\begin{proof}
By Lemma \ref{lem:am-gm} we have
\[ \sum_{p\leq N^{\alpha}}\frac{\log p}{|A\pmod p|}\geq\frac{1}{\alpha-c}\sum_{p\leq N^{\alpha}}\frac{\log p}{p}>\log N+1. \]
The result then follows from Lemma \ref{lem:gallagher}.
\end{proof}

The following two lemmas provide an improved larger sieve assuming some non-uniformity. The first one roughly says that if the bound $|A|\ll N^{\alpha}$ is almost sharp, then $|A\pmod p|\approx\alpha p$ for most of the primes $p$.

\begin{lemma}[Larger sieve with non-uniform sieving size]\label{lem:non-uniform-size}
Let $A\subset [N]$ be a subset and $\alpha\in (0,1)$ be real. Let $\mathcal{P}$ be a subset of the primes up to $N^{\alpha}$ with
\[ \sum_{p\in\mathcal{P}}\frac{\log p}{p}\geq c\sum_{p\leq N^{\alpha}}\frac{\log p}{p} \]
for some $c>0$. Let $c'>0$ be sufficiently small depending on $c$. Suppose that
\[ \sum_{p\leq N^{\alpha}}\frac{\log p}{p}\cdot\frac{|A\pmod p|}{p}\leq (\alpha+c')\sum_{p\leq N^{\alpha}}\frac{\log p}{p} \]
and
\[ \sum_{p\in\mathcal{P}}\frac{\log p}{p}\cdot\frac{|A\pmod p|}{p}\leq (\alpha-c)\sum_{p\in\mathcal{P}}\frac{\log p}{p}. \]
Then $|A|\ll N^{\alpha-c'}$. Indeed, we may take $c'=c^3\alpha/100$.
\end{lemma}

\begin{proof}
Write
\[ S_1=\sum_{\substack{p\leq N^{\alpha-c'}\\ p\in\mathcal{P}}}\frac{\log p}{p}, \ \ S_2=\sum_{\substack{p\leq N^{\alpha-c'}\\ p\notin\mathcal{P}}}\frac{\log p}{p}. \]
Then
\[  S_1+S_2= (\alpha-c')\log N+O(1), \ \ S_1>\frac{1}{2}c\alpha\log N. \]
Moreover, since
\[ \sum_{N^{\alpha-c'}<p\leq N^{\alpha}}\frac{\log p}{p}\leq 2c'\log N\leq \frac{1}{4}c^2\alpha\log N\leq \frac{1}{2}c^2\sum_{p\leq N^{\alpha}}\frac{\log p}{p} \leq \frac{1}{2}c\sum_{p\in\mathcal{P}}\frac{\log p}{p}, \]
we also have
\[ S_1>\left(1-\frac{c}{2}\right)\sum_{p\in\mathcal{P}}\frac{\log p}{p}. \]
Define $\kappa_1,\kappa_2$ by
\[ \sum_{\substack{p\leq N^{\alpha-c'}\\ p\in\mathcal{P}}}\frac{\log p}{p}\cdot\frac{|A\pmod p|}{p}=\kappa_1S_1,\ \ 
\sum_{\substack{p\leq N^{\alpha-c'}\\ p\not\in\mathcal{P}}}\frac{\log p}{p}\cdot\frac{|A\pmod p|}{p}=\kappa_2S_2. \]
Then
\[ \kappa_1S_1+\kappa_2S_2<\alpha(\alpha+c')\log N+O(1), \ \ \kappa_1\leq\frac{\alpha-c}{1-c/2}\leq\alpha-\frac{c}{2}. \]
We may further assume that
\[ \sum_{p\leq N^{\alpha-c'}}\frac{\log p}{p}\cdot\frac{|A\pmod p|}{p}=\kappa_1S_1+\kappa_2S_2>\alpha(\alpha-2c')\log N, \]
since otherwise we already have $|A|\ll N^{\alpha-c'}$ by Lemma \ref{lem:gallagher2}.
From these we can further deduce that
\[ (\kappa_2-\kappa_1)S_2=(\kappa_1S_1+\kappa_2S_2)-\kappa_1(S_1+S_2)>\alpha(\alpha-2c')\log N-\left(\alpha-\frac{c}{2}\right)\alpha\log N>\frac{1}{3}c\alpha\log N. \]
Since $S_2\leq\alpha\log N$, we then have
\[ \kappa_2-\kappa_1>\frac{c}{3}. \] 
By Lemma \ref{lem:am-gm} we have
\[ \sum_{\substack{p\leq N^{\alpha-c'}\\ p\in\mathcal{P}}}\frac{\log p}{|A\pmod p|}\geq\frac{S_1}{\kappa_1},\ \ \sum_{\substack{p\leq N^{\alpha-c'}\\ p\notin\mathcal{P}}}\frac{\log p}{|A\pmod p|}\geq\frac{S_2}{\kappa_2}. \]
In view of Lemma \ref{lem:gallagher}, it thus suffices to show that
\[ \frac{S_1}{\kappa_1}+\frac{S_2}{\kappa_2}>\log N+1. \]
Indeed, we have
\begin{align*}
\frac{S_1}{\kappa_1}+\frac{S_2}{\kappa_2} &=\frac{1}{\kappa_1S_1+\kappa_2S_2}\left[(S_1+S_2)^2+\frac{S_1S_2}{\kappa_1\kappa_2}(\kappa_1-\kappa_2)^2\right] \\
&\geq \frac{(S_1+S_2)^2+S_1\cdot (\kappa_2-\kappa_1)\cdot (\kappa_2-\kappa_1)S_2}{\kappa_1S_1+\kappa_2S_2}.
\end{align*}
Plugging in various bounds obtained so far we get
\[ \frac{S_1}{\kappa_1}+\frac{S_2}{\kappa_2}\geq \frac{(\alpha^2-2c'\alpha)(\log N)^2+\frac{1}{2}c\alpha\log N\cdot \frac{c}{3}\cdot \frac{1}{3}c\alpha\log N}{\alpha(\alpha+2c')\log N}\geq \frac{\alpha(\alpha+3c')(\log N)^2}{\alpha(\alpha+2c')\log N} \]
by our choice of $c'$, as desired.
\end{proof}

The following lemma roughly says that, if the bound $|A|\ll N^{\alpha}$ is almost sharp, then $A$ must be equidistributed modulo $p$ for most of the primes $p$. See also Lemma 2.4 in \cite{GH14}.

\begin{lemma}[Larger sieve with non-uniform fiber]\label{lem:non-uniform-fiber}
Let $c>0$ be given and $c'>0$ be sufficiently small depending on $c$. Let $A\subset [N]$ be a subset and $\alpha\in (0,1)$ be real. Let $\mathcal{P}$ be a subset of the primes up to $N^{\alpha}$ such that
\[ \sum_{\substack{p\notin\mathcal{P}\\ p\leq N^{\alpha}}}\frac{\log p}{p}<c'\log N. \]
Assume that for each $p\in\mathcal{P}$ we have $|A\pmod p|
\leq\alpha p$. Let
\[ \mathcal{P}_{\unif}=\left\{p\in\mathcal{P}:\sum_{r\pmod p}|\{a\in A:a\equiv r\pmod p\}|^2\leq\left(\frac{1}{\alpha}+c\right)\frac{|A|^2}{p}\right\}. \]
If
\[ \sum_{\substack{p\notin\mathcal{P}_{\unif}\\ p\leq N^{\alpha}}}\frac{\log p}{p}>c\log N, \]
then $|A|\ll N^{\alpha-c'}$. Indeed, one can take $c'=c^2\alpha/10$.
\end{lemma}

\begin{proof}
Let $Q=N^{\alpha-c'}$.
As in the proof of the ordinary larger sieve, consider the quantity
\[ I=\sum_{p\leq Q}\sum_{\substack{a,b\in A\\ a\neq b}}1_{p\mid a-b}\log p \]
 By first summing over $p$ and then summing over $a,b$, we have the upper bound $I\leq |A|^2\log N$. On the other hand, by first summing over $a,b$ and then summing over $p$, we have the lower bound
\[ I\geq\sum_{\substack{p\in\mathcal{P}\\ p\leq Q}}\sum_{r\pmod p}|\{a\in A:a\equiv r\pmod p\}|^2\log p-O(|A|Q). \]
For $p\in\mathcal{P}$, by Cauchy-Schwarz we have
\[ \sum_{r\pmod p}|\{a\in A:a\equiv r\pmod p\}|^2\geq\frac{|A|^2}{|A\pmod p|}\geq\frac{|A|^2}{\alpha p}. \]
For $p\in\mathcal{P}\setminus\mathcal{P}_{\unif}$, by definition we have
\[ \sum_{r\pmod p}|\{a\in A:a\equiv r\pmod p\}|^2\geq\left(\frac{1}{\alpha}+c\right)\frac{|A|^2}{p}. \]
Hence,
\[ I\geq |A|^2\left(\frac{1}{\alpha}\sum_{\substack{p\in\mathcal{P}\\ p\leq Q}}\frac{\log p}{p}+c\sum_{\substack{p\in\mathcal{P}\setminus\mathcal{P}_{\unif}\\ p\leq Q}}\frac{\log p}{p}\right)-O(|A|Q)\geq |A|^2\left[\frac{1}{\alpha}(\alpha-3c')+\frac{c^2}{2}\right]\log N-O(|A|Q). \]
By our choice of $c'$ we have
\[ I>\left(1+\frac{c^2}{6}\right)|A|^2\log N-O(|A|Q). \]
Combining this with the upper bound for $I$, we conclude that $|A|\ll Q$ as desired.
\end{proof}

The uniformity condition for $p\in\mathcal{P}_{\unif}$ will be used in the following way.

\begin{lemma}[uniform fiber property]\label{lem:fiber}
Let the notations be as in Lemma \ref{lem:non-uniform-fiber}. If $p\in\mathcal{P}_{\unif}$, then $|A\pmod p|\geq\alpha (1-c\alpha)p$, and furthermore
\[ |\{a\in A:a\equiv r\pmod p\}|\geq \left(\frac{1}{\alpha}-c^{1/3}\right)\frac{|A|}{p} \]
for all but at most $c^{1/3}p$ residues $r\in A\pmod p$.
\end{lemma}

\begin{proof}
For $r\in A\pmod p$, let $x_r=|\{a\in A:a\equiv r\pmod p\}|$. Then,
\[ \sum_r x_r=|A|,\ \ \sum_{r}x_r^2\leq\left(\frac{1}{\alpha}+c\right)\frac{|A|^2}{p}. \]
By Cauchy-Schwarz,
\[ |A\pmod p|\geq \frac{|A|^2}{\sum_r x_r^2}\geq \frac{|A|^2}{\left(\frac{1}{\alpha}+c\right)\frac{|A|^2}{p}}\geq \alpha(1-c\alpha)p. \]
Furthermore,
\[ \sum_r \left(x_r-\frac{|A|}{\alpha p}\right)^2\leq\left(\frac{1}{\alpha}+c\right)\frac{|A|^2}{p}-\frac{2|A|^2}{\alpha p}+\frac{|A|^2}{\alpha^2p^2}\cdot |A\pmod p|\leq \frac{c|A|^2}{p}. \]
If $x_r< (1/\alpha-c^{1/3})|A|/p$, then
\[ \left(x_r-\frac{|A|}{\alpha p}\right)^2>\frac{c^{2/3}|A|^2}{p^2}, \]
and thus the number of such $r$'s is at most $c^{1/3}p$.
\end{proof}

\subsection{Some results from additive combinatorics}

Let $A,B\subset \Z/p\Z$. We will be interested in lower bounds for the size of the sumset $A+B$. More generally, we need lower bounds for the number of elements in the sumset that has many representations as the sum of two elements from $A$ and $B$.

\begin{lemma}[Cauchy-Davenport-Chowla]\label{lem:cdc}
Let $p$ be prime and let $A,B\subset\Z/p\Z$ be subsets. Then
\[ |A+B|\geq\min(p,|A|+|B|-1). \]
\end{lemma}

\begin{proof}
See Theorem 5.4 in \cite{TV06}.
\end{proof}

\begin{lemma}[robust Cauchy-Davenport-Chowla]\label{lem:gr}
Let $p$ be prime and let $A,B\subset\Z/p\Z$ be subsets. Let $K>0$ be real. Let $S\subset A+B$ be the set of elements in $A+B$ with at least $K$ representations $a+b$ ($a\in A,b\in B$); in other words,
\[ S=\{s\in A+B:|A\cap (s-B)|\geq K\}. \]
If $|A|,|B|\geq\sqrt{Kp}$, then
\[ |S|\geq\min(p,|A|+|B|-1)-3\sqrt{Kp}. \]
\end{lemma}

\begin{proof}
See Corollary 6.2 in \cite{GR05}.
\end{proof}

We also need the following structural theory of sets with (very) small doubling.

\begin{lemma}\label{lem:gryn}
Let $p$ be prime and let $A,B\subset\Z/p\Z$ be subsets. Let $c>0$ be sufficiently small. If $|A|,|B|\geq cp+3$, and moreover
\[ |A+B|\leq\min(|A|+|B|-1+cp,(1-c)p-3), \]
then $A$ and $B$ are contained in arithmetic progressions of length at most $|A|+cp$ and $|B|+cp$, respectively.
\end{lemma}

\begin{proof}
See Theorem 21.8 in \cite{Gry13}.
\end{proof}

\section{Improved larger sieve assuming additive structures}\label{sec:small-doubling}

\begin{proposition}[larger sieve with small doubling]\label{prop:small-doubling}
Let $A\subset [N]$ be a subset and $\alpha\in (0,1/2]$ be real. Assume that $|A-A|\leq |A|^{3/2-c}$ for some $c>0$. Let $c'>0$ be sufficiently small depending on $c$. Let $\mathcal{P}$ be a subset of the primes up to $N^{\alpha}$ such that
\[ \sum_{\substack{p\notin\mathcal{P}\\ p\leq N^{\alpha}}}\frac{\log p}{p}<c'\log N. \]
If $|A\pmod p|\leq\alpha p$ for each $p\in\mathcal{P}$, then $|A|\ll N^{\alpha-c'}$. Indeed, we may take $c'=10^{-50}(c\alpha)^{25}$.
\end{proposition}

\begin{proof}
Write $A_p=A\pmod p$. Let $\epsilon>0$ be a small parameter depending on $c$ to be chosen later (Indeed, we may take $\epsilon=(c\alpha/100)^2$ and then $c'=\epsilon^{12}\alpha/10$). Let $R=|A|^{(1+c)/2}$ and $Q=\min(|A|^{1/2+c/4},N^{\alpha})$. For any $h\in\Z$ and $p\in\mathcal{P}$, let $\nu_p(h)=|A_p\cap (A_p+h)|$, the number of ways to write $h\pmod p$ as the difference of two elements of $A_p$. Let $\nu(h)=|A\cap (A+h)|$, the number of ways to write $h$ as the difference of two elements of $A$. Consider the quantity
\[ J=\sum_{a,b\in A}\sum_{\substack{p\in\mathcal{P}\\ p\leq Q\\ \nu_p(a-b)\geq\epsilon p}}\frac{\log p}{\nu_p(a-b)}. \]

We first obtain an upper bound for $J$. For those pairs $a,b\in A$ with $\nu(a-b)>R$, we apply the larger sieve (Lemma \ref{lem:gallagher}) to the set $A\cap (A+a-b)$. Since $|A\cap (A+a-b)|>R$, and the residues $A\cap (A+a-b)\pmod p$ lie in the intersection $A_p\cap (A_p+a-b)$, which has size $\nu_p(a-b)$, we conclude that
\[ \sum_{\substack{p\in\mathcal{P}\\ p\leq Q}}\frac{\log p}{\nu_p(a-b)}\leq\log N+1. \]
Hence the contribution to $J$ from these pairs is at most $|A|^2(\log N+1)$. 

For those pairs $a,b\in A$ with $\nu(a-b)\leq R$, the inner sum over $p$ can be bounded trivially by $O(\log N)$. The number of such pairs is
\[ \sum_{\substack{h\in A-A\\ \nu(h)\leq R}}\nu(h)\leq R|A-A|=o(|A|^2)  \]
by our choice of $R$. Hence
\[ J\leq (1+o(1))|A|^2\log N. \]

We now seek for a lower bound for $J$. By Lemma \ref{lem:non-uniform-fiber}, either $|A|\ll N^{\alpha-c'}$ and we are done, or else $A$ has uniform fiber over almost all primes $p\in\mathcal{P}$. More precisely, let
\[ \mathcal{P}_{\unif}=\left\{p\in\mathcal{P}:\sum_{r\pmod p}|\{a\in A:a\equiv r\pmod p\}|^2\leq\left(\frac{1}{\alpha}+\epsilon^6\right)\frac{|A|^2}{p}\right\}. \]
Then,
\[ \sum_{\substack{p\notin\mathcal{P}_{\unif}\\ p\leq N^{\alpha}}}\frac{\log p}{p}\leq \epsilon^6\log N. \]

Fix a prime $p\in\mathcal{P}_{\unif}$. By Lemma \ref{lem:fiber}, we have $|A_p|\geq\alpha (1-\epsilon^6\alpha)p$, and there exists a subset $A_p'\subset A_p$ with $|A_p\setminus A_p'|\leq \epsilon^2 p$, such that
\[ |\{a\in A:a\equiv r\pmod p\}|\geq\left(\frac{1}{\alpha}-\epsilon^2\right)\frac{|A|}{p} \]
for each $r\in A_p'$. For $r\in A_p$ write $x_r=|\{a\in A:a\equiv r\pmod p\}|$. We have
\[ \sum_{\substack{a,b\in A\\ \nu_p(a-b)\geq\epsilon p}}\frac{1}{\nu_p(a-b)}\geq \sum_{\substack{r,s\in A_p'\\ \nu_p(r-s)\geq\epsilon p}}\frac{x_rx_s}{\nu_p(r-s)}\geq \frac{1}{\alpha^2}\left(1-2\epsilon^2\right)\frac{|A|^2}{p^2}\sum_{\substack{r,s\in A_p'\\ \nu_p(r-s)\geq\epsilon p}}\frac{1}{\nu_p(r-s)}.  \]
For any $h\in \Z/p\Z$, the number of ways to write $h$ as the difference of two elements in $A_p'$ is at least $\nu_p(h)-2\epsilon^2 p$. Hence
\[ \sum_{\substack{r,s\in A_p'\\ \nu_p(r-s)\geq\epsilon p}}\frac{1}{\nu_p(r-s)}\geq \sum_{\substack{h\in\Z/p\Z\\ \nu_p(h)\geq\epsilon p}}\frac{\nu_p(h)-2\epsilon^2 p}{\nu_p(h)}\geq (1-2\epsilon)\cdot |\{h\in\Z/p\Z:\nu_p(h)\geq\epsilon p\}|. \]
By Lemma \ref{lem:gr}, the number of differences $h\in\Z/p\Z$ with $\nu_p(h)\geq\epsilon p$ is at least
\[ \min(p,2|A_p|-1)-3\epsilon^{1/2}p>(2\alpha-5\epsilon^{1/2})p.  \]
Consequently,
\[ \sum_{\substack{a,b\in A\\ \nu_p(a-b)\geq\epsilon p}}\frac{1}{\nu_p(a-b)}\geq\frac{1}{\alpha^2}(1-2\epsilon^2)\frac{|A|^2}{p^2}(1-2\epsilon)(2\alpha-5\epsilon^{1/2})p\geq\frac{2}{\alpha^2}(\alpha-8\epsilon^{1/2})\frac{|A|^2}{p}.  \]
Summing over all $p\in\mathcal{P}_{\unif}$ and $p\leq Q$ we get
\[ J\geq \frac{2}{\alpha^2}(\alpha-8\epsilon^{1/2})|A|^2\sum_{\substack{p\in\mathcal{P}_{\unif}\\ p\leq Q}}\frac{\log p}{p}\geq \frac{2}{\alpha^2}(\alpha-8\epsilon^{1/2})|A|^2(\log Q-\epsilon^6\log N-O(1)). \]
Combining this with the upper bound for $J$ previously obtained, we deduce that
\[ \frac{\log Q}{\log N}\leq \frac{\alpha^2}{2(\alpha-8\epsilon^{1/2})}+\epsilon^6+o(1) \leq
\frac{\alpha}{2}+10\epsilon^{1/2}\leq \alpha\left(\frac{1}{2}+\frac{c}{8}\right) \]
by our choice of $\epsilon$. This means that $Q\leq N^{\alpha(1/2+c/8)}$. By our choice of $Q$, this implies that $|A|\ll N^{\alpha(1-c^2/16)}$, as desired.
\end{proof}

\section{A variant of the large sieve}\label{sec:large}

In this section we prove Theorem \ref{thm:large}. We start by proving its dual form. Fix parameters $P$ and $N$. Let $Y$ be the space of functions on the set $\{a/p:p\leq P,1\leq a\leq p-1\}$, and $X$ be the space of functions on $[N]$. We will use $g$ to denote a typical function in $Y$ and write $g(a/p)=g_p(a)$. Let $L:Y\rightarrow X$ be the linear operator defined by
\[ L(g)(n)=\sum_{p\leq P}\sum_{a=1}^{p-1}g_p(a)e\left(\frac{an}{p}\right), \]
for $g\in Y$ and $1\leq n\leq N$. Equip $X$ with the $L^{2k}$-norm and equip $Y$ with a norm $\|\cdot\|_Y$ that is the sum of two norms:
\[ \|g\|_Y=\|g\|_{Y_1}+\|g\|_2, \]
where $\|g\|_2$ is the usual $L^2$-norm
\[ \|g\|_2^2=\sum_{p\leq P}\sum_{a=1}^{p-1}|g_p(a)|^2, \]
and the norm $\|\cdot\|_{Y_1}$ is defined by
\[ \|g\|_{Y_1}^{2k}=\sum_{p\leq P}\left(\sum_{a=1}^{p-1}|g_p(a)|^{2k/(2k-1)}\right)^{2k-1}=\sum_{p\leq P} \|g_p\|_{2k/(2k-1)}^{2k}. \]

\begin{proposition}[Large sieve variant, dual form]\label{prop:large-dual}
For each prime $p\leq P$, let $g_p:(\Z/p\Z)^*\rightarrow\C$ be an arbitrary function. For any positive integer $k$ we have
\[ \sum_{n\leq N}\left|\sum_{p\leq P}\sum_{a=1}^{p-1}g_p(a)e\left(\frac{an}{p}\right)\right|^{2k}\ll (N+P^{2k})\left(\sum_{p\leq P}\|g_p\|_{2k/(2k-1)}^{2k}+\|g\|_2^{2k}\right). \]
In other words, we have $\|L\|_{\text{op}}\ll N^{1/2k}+P$.
\end{proposition}

\begin{proof}
After expanding out the $2k$th power, we can rewrite the left side as
\[ \sum_{p_1,a_1,\cdots,p_{2k},a_{2k}}g_{p_1}(a_1)\cdots g_{p_k}(a_k)\overline{g_{p_{k+1}}(a_{k+1})\cdots g_{p_{2k}}(a_{2k})}\sum_{n\leq N}e\left(n\left(\frac{a_1}{p_1}+\cdots+\frac{a_k}{p_k}-\frac{a_{k+1}}{p_{k+1}}-\cdots-\frac{a_{2k}}{p_{2k}}\right)\right). \]
For each $x\in\R/\Z$, let $S(x)$ be the collection of all tuples $(p_1,a_1,\cdots,p_k,a_k)$ with
\[ \frac{a_1}{p_1}+\cdots+\frac{a_k}{p_k}\in x+\Z. \]
If we write
\[ s(x)=\sum_{(p_1,a_1,\cdots,p_k,a_k)\in S(x)}g_{p_1}(a_1)\cdots g_{p_k}(a_k), \]
then the left side becomes
\[ \sum_{x,y\in\R/\Z}s(x)\overline{s(y)}\sum_{n\leq N}e(n(x-y))=\sum_{\substack{x\in\R/\Z\\ s(x)\neq 0}}\left|\sum_{n\leq N}s(x)e(nx)\right|^2. \]
The set of points $\{x:s(x)\neq 0\}$ is $\delta$-spaced with $\delta=P^{-2k}$. It thus follows from the traditional large sieve inequality (Theorem \ref{thm:large-old}) that the left side of the desired inequality is bounded by
\[  (N+P^{2k})\sum_x |s(x)|^2. \]
To bound $\sum_x |s(x)|^2$, expand out to get
\[ \sum_x |s(x)|^2\leq \sum |g_{p_1}(a_1)\cdots g_{p_{2k}}(a_{2k})|, \]
where the sum above is over tuples $(p_1,a_1,\cdots,p_{2k},a_{2k})$ satisfying
\[ \frac{a_1}{p_1}+\cdots+\frac{a_k}{p_k}-\frac{a_{k+1}}{p_{k+1}}-\cdots-\frac{a_{2k}}{p_{2k}}\in\Z. \]
Note that for any such tuple, we must have
\[ \sum_{\substack{1\leq i\leq k\\ p_i=p}}a_i-\sum_{\substack{k+1\leq i\leq 2k\\ p_i=p}}a_i\equiv 0\pmod p, \]
for any $p\in\{p_1,\cdots,p_{2k}\}$. Thus for fixed denominators $p_1,\cdots,p_{2k}$, the sum over $a_1,\cdots,a_{2k}$ is bounded by
\[ \prod_{p\in\{p_1,\cdots,p_{2k}\}}\left(\sum_{\substack{a_1,\cdots,a_{n_p}\\ \pm a_1\pm\cdots\pm a_{n_p}\equiv 0\pmod p}}|g_p(a_1)\cdots g_p(a_{n_p})|\right), \]
with an appropriate choice from those $\pm$ signs, where $n_p$ is the number of multiplicities of $p$ in $\{p_1,\cdots,p_{2k}\}$. Note that we must have $n_p\geq 2$. By Young's inequality, the inner sum is bounded by
\[ \sum_{\substack{a_1,\cdots,a_{n_p}\\ \pm a_1\pm\cdots\pm a_{n_p}\equiv 0\pmod p}}|g_p(a_1)\cdots g_p(a_{n_p})|\leq\|g_p\|_{n_p/(n_p-1)}^{n_p}. \]
We conclude that for fixed multiplicities $(n_1,\cdots,n_d)$ with $n_1,\cdots,n_d\geq 2$ and $n_1+\cdots+n_d=2k$, the sum over those $(p_1,\cdots,p_{2k})$ satisfying the given multiplicities is bounded by
\[ \prod_{i=1}^d\left[\sum_p\|g_p\|_{n_i/(n_i-1)}^{n_i}\right]. \]
It remains to show that this is always dominated by one of the two extreme cases: $d=1,n_1=2k$, or $d=k,n_1=\cdots=n_k=2$:
\[ \prod_{i=1}^d\left[\sum_p \|g_p\|_{n_i/(n_i-1)}^{n_i}\right]\leq \sum_p \|g_p\|_{2k/(2k-1)}^{2k}+\|g\|_2^{2k}. \]
To prove this, we may assume that $k\geq 2$. By H\"{o}lder's inequality,
\[ \|g_p\|_{n_i/(n_i-1)}^{n_i}\leq \|g_p\|_{2k/(2k-1)}^{(n_i-2)k/(k-1)}\|g_p\|_2^{(2k-n_i)/(k-1)}.  \]
By a further application of H\"{o}lder's inequality,
\[ \sum_p \|g_p\|_{n_i/(n_i-1)}^{n_i} \leq \left(\sum_p \|g_p\|_{2k/(2k-1)}^{2k}\right)^{\frac{n_i-2}{2k-2}}\left(\sum_p \|g_p\|_2^2\right)^{\frac{2k-n_i}{2k-2}}. \]
It follows that
\[ \prod_{i=1}^d\left[\sum_p \|g_p\|_{n_i/(n_i-1)}^{n_i}\right]\leq \left[\sum_p \|g_p\|_{2k/(2k-1)}^{2k}\right]^{\frac{k-d}{k-1}} \left[\sum_p \|g_p|_2^2\right]^{\frac{k(d-1)}{k-1}}. \]
This completes the proof.
\end{proof}

Now we dualize. The dual $X^*$ can be identified as the space of functions on $[N]$, equipped with the $L^{2k/(2k-1)}$-norm. Finding the dual $Y^*$ is trickier. As a set, $Y^*$ can be identified with $Y$. Let $\|\cdot\|_{Y^*}$ and $\|\cdot\|_{Y_1^*}$ be the dual norm for $\|\cdot\|_Y$ and $\|\cdot\|_{Y_1}$, respectively. It is easy to find the dual norm for $\|\cdot\|_{Y_1}$:
\[ \|h\|_{Y_1^*}^{2k/(2k-1)}=\sum_{p\leq P} \|h_p\|_{2k}^{2k/(2k-1)}. \]
For the dual norm for $\|\cdot\|_Y$, we have the following lower bound.

\begin{lemma}\label{lem:dual-norm}
Let the notations be as above. For $h\in Y^*$ we have
\[ 2\|h\|_{Y^*}\geq P^{-(k-1)/2k}\left(\sum_{p\leq P} \|h_p\|_2^{2k/(2k-1)}\right)^{(2k-1)/2k}. \]
\end{lemma}

\begin{proof}
We use the following interpretation for the dual of the sum of two norms: 
\[ \|h\|_{Y^*}=\inf\{\max(\|h_1\|_{Y_1^*},\|h_2\|_2):h_1,h_2\in Y^*,h=h_1+h_2\}. \]
To see this, let $Y\oplus Y$ be the space equipped with the norm $\|(y_1,y_2)\|=\|y_1\|_{Y_1}+\|y_2\|_2$. The dual norm on $Y\oplus Y$ is easily seen to be $\|(y_1,y_2)\|_*=\max(\|y_1\|_{Y_1^*},\|y_2\|_2)$. View $Y$ as a subspace of $Y\oplus Y$ via the isometric diagonal embedding $i:Y\rightarrow Y\oplus Y$ with $i(y)=(y,y)$. 
For any linear functional $h\in Y^*$, Hahn-Banach theorem says that it can be extended to a linear functional $(h_1,h_2)\in (Y\oplus Y)^*$ with the same norm:
\[ \|h\|_{Y^*}=\|(h_1,h_2)\|_*=\max(\|h_1\|_{Y_1^*},\|h_2\|_2). \]
Since $(h_1,h_2)$ restricts to $h$ on $Y$, we have $h=h_1+h_2$. This proves that the left side is at least as large as the right side. For the other direction, note that for any $g\in Y$,
\[ \langle h,g\rangle=\langle h_1,g\rangle+\langle h_2,g\rangle\leq \|h_1\|_{Y_1^*}\|g\|_{Y_1}+\|h_2\|_2\|g\|_2\leq \|g\|_Y\cdot\max(\|h_1\|_{Y_1^*},\|h_2\|_2). \]

Now let $\|\cdot\|_{Y'}$ be the norm
\[ \|h\|_{Y'}=P^{-(k-1)/2k}\left(\sum_{p\leq P} \|h_p\|_2^{2k/(2k-1)}\right)^{(2k-1)/2k}. \]
By H\"{o}lder's inequality, the norm $\|\cdot\|_{Y'}$ is a lower bound for both the norm $\|\cdot\|_{Y_1^*}$ and the $L^2$-norm. Indeed, for the $\|\cdot\|_{Y_1^*}$ norm, since
\[ p^{(k-1)/2k}\|h_p\|_{2k}\geq \|h_p\|_2, \]
we have
\[ \|h\|_{Y_1^*}^{2k/(2k-1)}=\sum_{p\leq P} \|h_p\|_{2k}^{2k/(2k-1)}\geq\sum_{p\leq P} p^{-(k-1)/(2k-1)}\|h_p\|_2^{2k/(2k-1)}, \]
as desired. For the $L^2$-norm, we have
\[ \|h\|_2^2=\sum_{p\leq P}\|h_p\|_2^2\geq P^{-(k-1)/k}\left(\sum_{p\leq P}\|h_p\|_2^{2k/(2k-1)}\right)^{(2k-1)/k},  \]
as desired.

For any decomposition $h=h_1+h_2$, we have
\[ \|h\|_{Y'}\leq \|h_1\|_{Y'}+\|h_2\|_{Y'}\leq \|h_1\|_{Y_1^*}+\|h_2\|_2\leq 2\max(\|h_1\|_{Y_1^*},\|h_2\|_2). \]
The proof is completed by taking the infimum over all such decompositions.
\end{proof}

\begin{remark}
This lower bound is not sharp. Using the fact that
\[ \|h\|_{Y^*}=\sup_{0\neq g\in Y}\frac{\langle h,g\rangle}{\|g\|_{Y_1}+\|g\|_2}, \]
one can obtain two other lower bounds by taking the test function $g$ to be the dual of $h$ under the norm $\|\cdot\|_{Y_1}$ or the $L^2$-norm. It can be checked that none of these three lower bounds always beats any other. The lower bound in this lemma takes the simplest form and is enough for our applications.
\end{remark}

\begin{proof}[Proof of Theorem \ref{thm:large}]
We follow the notations set up earlier in this section. Consider the dual map $L^*:X^*\rightarrow Y^*$. Recall that $X^*$ is the space of functions on $[N]$ equipped with the $L^{2k/(2k-1)}$-norm. For any function $f\in X^*$, from the definition of $L$ we see that
\[ L^*f(a/p)=\sum_{n=1}^N f(n)e(-an/p)=\hat{f}(a/p). \]
Since $\|L^*\|=\|L\|\ll N^{1/2k}+P$, we conclude that for any function $f$,
\[ \|L^*f\|_{Y^*}\ll (N^{1/2k}+P)\|f\|_{2k/(2k-1)}. \]
The desired inequality follows by combining this with the lower bound in Lemma \ref{lem:dual-norm}.
\end{proof}

\section{Sieving arithmetic progressions}\label{sec:sieve-ap}

In this section we prove Theorem \ref{thm:sieve-ap-simple}, restated here with the slight generalization of allowing a small set of exceptional primes.

\begin{theorem}[Sieving arithmetic progressions]\label{thm:sieve-ap-larger}
Let $A\subset [N]$ be a subset. Let $\alpha\in [1/3,1/2]$ be real and $c>0$ be sufficiently small (not depending on $\alpha$). Let $\mathcal{P}$ be a subset of the primes up to $N^{\alpha}$ such that
\[ \sum_{\substack{p\notin\mathcal{P}\\ p\leq N^{\alpha}}}\frac{\log p}{p}<c\log N. \]
If, for each prime $p\in\mathcal{P}$, the residues $A\pmod p$ lie in an arithmetic progression $S_p\subset\Z/p\Z$ of length at most $\alpha p$, then $|A|\ll N^{\alpha-c}$.
\end{theorem}

If $\alpha$ is bounded away from $1/3$, this follows easily from Proposition \ref{prop:small-doubling}, since we have $|A+A|\ll N^{1/2}$ by the large sieve. Theorem \ref{thm:sieve-ap-larger} will be applied with $\alpha$ slightly above $1/3$, and it is crucial that the constant $c$ in the statement does not depend on how close $\alpha$ is to $1/3$. To handle the case when $\alpha\approx 1/3$, we need a bound better than $|A+A|\ll N^{1/2}$. This is achieved by the following:

\begin{theorem}[Sieving arithmetic progressions in the large sieve regime]\label{thm:sieve-ap-large}
Let $A\subset [N]$ be a subset. Let $k$ be a positive integer and $\epsilon>0$ be real. Let $c>0$ be sufficiently small depending on $k$ and $\epsilon$. Let $\mathcal{P}$ be a subset of the primes up to $N^{1/2k}$ satisfying
\[ \sum_{\substack{p\notin\mathcal{P}\\ p\leq N^{1/2k}}}\frac{\log p}{p}<c\log N. \]
If, for each $p\in\mathcal{P}$, the residues $A\pmod p$ lie in an arithmetic progression $S_p\subset\Z/p\Z$ of length at most $(1-\epsilon)p$, then $|A|\ll N^{1/2-c}$.
\end{theorem}

This is a strengthening of Theorem 1.3 in \cite{GH14}, in that we only have information for primes up to $N^{1/2k}$ (this should be compared with Corollary \ref{cor:vaughan}). To adapt the argument in \cite{GH14}, we need to replace Lemma 5.1 there by the following lemma, which is a consequence of Theorem \ref{thm:large}. In this section we will simply state and prove this lemma. See appendix for details on the rest of the arguments that are more or less identical with those in Section 5 of \cite{GH14}.

\begin{lemma}[Lifting additive energy]\label{lem:lift}
Let the notation be as in Theorem \ref{thm:sieve-ap-large}. We have
\[ E(A)\gg |A|^3\cdot\frac{|A|}{\sqrt{N}}\cdot\left(\frac{|\mathcal{P}|}{N^{1/2k}}\right)^{2k-1}. \]
\end{lemma}

Here $E(A)$ is the additive energy of $A$, defined by the number of quadruples $(a_1,a_2,a_3,a_4)\in A\times A\times A\times A$ with $a_1+a_2=a_3+a_4$.

\begin{proof}
We will apply Theorem \ref{thm:large} with $f=1_A*1_A$. For each $p\in\mathcal{P}$, by Lemma 4.1 of \cite{GH14} there exists $1\leq a\leq p-1$ such that $|\widehat{1_A}(a/p)|\gg |A|$. Hence for $p\in\mathcal{P}$, we have $I_p(f)\gg |A|^4$. By H\"{o}lder's inequality,
\begin{align*} 
\sum_{n=1}^N|1_A*1_A(n)|^{2k/(2k-1)} &\leq \left(\sum_{n=1}^N 1_A*1_A(n)\right)^{(2k-2)/(2k-1)}\left(\sum_{n=1}^N 1_A*1_A(n)^2\right)^{1/(2k-1)} \\
&\leq |A|^{4(k-1)/(2k-1)}E(A)^{1/(2k-1)}.
\end{align*}
Hence Theorem \ref{thm:large} with $P=N^{1/2k}$ gives
\[ |\mathcal{P}||A|^{4k/(2k-1)}\ll P^{(k-1)/(2k-1)}N^{1/(2k-1)}|A|^{4(k-1)/(2k-1)}E(A)^{1/(2k-1)}. \]
This simplifies to the desired inequality.
\end{proof}

\begin{proof}[Proof of Theorem \ref{thm:sieve-ap-larger}]
In view of Theorem \ref{thm:sieve-ap-large} we may assume that $\alpha\leq 1/2-c$ for some small $c>0$. Consider the difference set $A-A$. For each prime $p\in\mathcal{P}$, the residues $(A-A)\pmod p$ lie in an arithmetic progression of length at most $2\alpha p\leq (1-2c)p$. By Theorem \ref{thm:sieve-ap-large} applied to $A-A$ (with $k=2$ say), we have $|A-A|\ll N^{1/2-c'}$ for some small $c'>0$ depending on $c$. Hence either $|A|\ll N^{\alpha-c'/2}$ and we are done, or else $|A-A|\ll |A|^{(1-2c')/(2\alpha-c')}\ll |A|^{(3-c')/2}$ since $\alpha\geq 1/3$, and the conclusion follows from Proposition \ref{prop:small-doubling}.
\end{proof}

\section{The inverse ternary Goldbach problem}\label{sec:proof-main}

\begin{proposition}\label{prop:main-weak}
Let $\epsilon>0$. If $A_1,A_2,A_3\subset [N]$ be subsets with $|A_1|,|A_2|,|A_3|\geq N^{1/3+\epsilon}$, then $A_1+A_2+A_3$ contains a composite number.
\end{proposition}

\begin{proof}
See Theorem 3 in \cite{PSS88}.
\end{proof}

\begin{proof}[Proof of Theorem \ref{thm:main}]
Let $A_1,A_2,A_3\subset [N]$ be subset with $|A_1|,|A_2|,|A_3|\geq N^{1/3-c}$. Suppose, for the purpose of contradiction, that $a_1+a_2+a_3$ is prime for any $a_1\in A_1,a_2\in A_2,a_3\in A_3$. 
By Proposition \ref{prop:main-weak}, we have $\min(|A_1\cap [N^{0.6}]|,|A_2\cap [N^{0.6}]|,|A_3\cap [N^{0.6}]|)\ll N^{0.21}$. We may therefore assume that all elements of $A_1+A_2+A_3$ are larger than $N^{0.6}$. Hence for all primes $p\leq N^{0.6}$ we have by Lemma \ref{lem:cdc},
\[ |A_1\pmod p|+|A_2\pmod p|+|A_3\pmod p|\leq p+1. \]

Let $c'>0$ be sufficiently small depending on $c$. By the larger sieve (Lemma \ref{lem:gallagher2}) we may assume that
\[ \sum_{p\leq N^{1/3-c'}}\frac{\log p}{p}\cdot\frac{|A_i\pmod p|}{p}\geq \left(\frac{1}{3}-2c'\right)\sum_{p\leq N^{1/3-c'}}\frac{\log p}{p} \]
for each $i=1,2,3$, since otherwise we would have $|A_i|\ll N^{1/3-c'}$ as desired. Note that the sum of the left side above over $i=1,2,3$ is at most
\[  \sum_{p\leq N^{1/3-c'}}\frac{\log p}{p}\cdot\frac{p+1}{p}=\sum_{p\leq N^{1/3-c'}}\frac{\log p}{p}+O(1). \]
It follows that
\[ \sum_{p\leq N^{1/3-c'}}\frac{\log p}{p}\cdot\frac{|A_i\pmod p|}{p}\leq \left(\frac{1}{3}+5c'\right)\sum_{p\leq N^{1/3-c'}}\frac{\log p}{p} \]
for each $i=1,2,3$. Let $\mathcal{P}_i$ be the set of primes $p\leq N^{1/3-c'}$ with $|A_i\pmod p|\leq (1/3-2c)p$, and let $\mathcal{P}$ be the set of primes $p\leq N^{1/3-c'}$ outside $\mathcal{P}_1,\mathcal{P}_2,\mathcal{P}_3$. By Lemma \ref{lem:non-uniform-size} we may assume that
\[ \sum_{p\in\mathcal{P}_i}\frac{\log p}{p}<c\sum_{p\leq N^{1/3-c'}}\frac{\log p}{p}<\frac{c}{3}\log N, \]
since otherwise we would have $|A_i|\ll N^{1/3-c'}$ as desired. Hence
\[ \sum_{\substack{p\notin\mathcal{P}\\ p\leq N^{1/3+15c}}}\frac{\log p}{p}\leq \sum_{p\in\mathcal{P}_1}\frac{\log p}{p}+\sum_{p\in\mathcal{P}_2}\frac{\log p}{p}+\sum_{p\in\mathcal{P}_3}\frac{\log p}{p}+\sum_{N^{1/3-c'}<p\leq N^{1/3+15c}}\frac{\log p}{p}  <20c\log N. \]
Note that for $p\in\mathcal{P}$, we have $|A_i\pmod p|>(1/3-2c)p$ for each $i$, and for any permutation $(i,j,k)$ of $\{1,2,3\}$, by Lemma \ref{lem:cdc} again we have
\[ |(A_i\pmod p)+(A_j\pmod p)|\leq p-|A_k\pmod p|\leq \left(\frac{2}{3}+2c\right)p. \]
By Lemma \ref{lem:gryn}, we conclude that each $A_i\pmod p$ for $p\in\mathcal{P}$ is contained in an arithmetic progression of length $(1/3+15c)p$. By Theorem \ref{thm:sieve-ap-larger}, we have $|A|\ll N^{1/3-c}$ for sufficiently small $c>0$.
\end{proof}

\section{Further remarks on the larger sieve}\label{sec:sharp}

In this last section we remark on the sharpness of the larger sieve. Recall that the larger sieve says that, if a subset $A\subset [N]$ occupies at most $\alpha p$ residue classes modulo $p$ for each $p\leq N^{\alpha}$, then $|A|\ll N^{\alpha}$. When $\alpha=1/2$ (in the large sieve regime), the bound $|A|\ll N^{1/2}$ is sharp by taking $A$ to be the set of squares up to $N$, so that $A\pmod p$ is the set of quadratic residues. 

For the rest of the discussion we are concerned with what happens for smaller $\alpha$, say $\alpha=1/d$ for some $d\geq 3$. Is the bound $|A|\ll N^{1/d}$ sharp? A tempting example to try is to take $A$ to be the set of $d$th powers up to $N$. Then $A\pmod p$ occupies $(p-1)/\text{gcd}(p-1,d)+1$ residue classes; this number is much larger than $p/d$ whenever $d\nmid p-1$. In the case when $d=3$, we have
\[ \sum_{p\leq Q}\frac{\log p}{|A\pmod p|}\sim\sum_{\substack{p\leq Q\\ p\equiv 1\pmod 3}}\frac{3\log p}{p}+\sum_{\substack{p\leq Q\\ p\equiv 2\pmod 3}}\frac{\log p}{p}\sim 2\log Q.  \]
Hence Gallagher's larger sieve only gives the upper bound $|A|\ll N^{1/2}$. Some variants of the larger sieve were obtained in \cite{CE04}, but they do not give better bounds in this situation. Thus we have no obvious evidence against the following conjecture.

\begin{conjecture}
Let $d\geq 3$. If $A\subset [N]$ is a subset that occupies at most $(1/d+o(1))p$ residue classes modulo $p$ for all primes $p\leq N^{1/d}$, then $|A|\ll N^{1/d-c}$ for some small $c=c(d)>0$.
\end{conjecture}

This conjecture, if true, would immediately lead to Theorem \ref{thm:main}, just as the usual larger sieve leads to Proposition \ref{prop:main-weak}. Not having a clue of proving this conjecture, we instead proved an improved larger sieve assuming additive structure, which is sufficient for our main theorem with the help of results from additive combinatorics.

As further evidences towards this conjecture, we remark that if $A$ is chosen to be the value set of any polynomial of degree $d$:
\[ A=\{P(x):x\in\Z\}\cap [N] \]
for some $P(x)\in\Z[x]$ with degree $d$, then the larger sieve applied to $A$ always leads to a bound worse than $|A|\ll N^{1/d}$. In fact, by a result in \cite{Gom88}, we know that if $d\nmid p-1$, then the value set of the polynomial $P(x)$ modulo $p$ is at least $(1/d+2/d^2)p$. Therefore
\[ \sum_{p\leq Q}\frac{\log p}{|A\pmod p|}\leq d\log Q\left[\frac{1}{\phi(d)}+\frac{d}{d+2}\left(1-\frac{1}{\phi(d)}\right)\right], \]
always smaller than $d\log Q$.

Finally, consider the related quantity
\[ \sum_{p\leq Q}\frac{\log p}{p}\cdot\frac{|A\pmod p|}{p}. \]
When $d=3$ and $A$ is the set of cubes up to $N$, this is about $(2/3)\log Q$. In other words, the average size of the value set of the polynomial $P(x)=x^3$ modulo primes is $2/3$.
We claim that the situation does not improve if $A$ is taken to be the value set of other cubic polynomials, in the sense that this average is always $2/3$ asymptotically. This is a consequence of the Chebotarev density theorem. Indeed, the quantity we are interested in is asymptotically equal to
\[ \frac{1}{R}\sum_{r=1}^R\sum_{\substack{p\leq Q\\ p\in\mathcal{P}_r}}\frac{\log p}{p}, \]
where $\mathcal{P}_r$ is the set of primes $p$ such that $P(x)\equiv r\pmod p$ has a solution, and $R\geq Q$ is large. The inner sum is the density of the primes (weighted by $\log p/p$) for which the congruence $P(x)-r\equiv 0\pmod p$ has a solution, and can be evaluated by the following result.

\begin{nonumtheorem}
Given an irreducible polynomial $P(x)\in \Z[x]$, let $L$ be the splitting field of $P$ over $\Q$ and $G=\text{Gal}(L/\Q)$. Let $K$ be the intermediate field $K=\Q[x]/(P(x))$ and let $H=\text{Gal}(L/K)$ be a subgroup of $G$. The density of the set of primes for which the congruence $P(x)\equiv 0\pmod p$ has a solution is equal to
\[ |G|^{-1}\left|\bigcup_{\sigma\in G}\sigma^{-1}H\sigma\right|. \]
\end{nonumtheorem}

\begin{proof}
See Theorem 2 in \cite{BB96}.
\end{proof}

Although the density above refers to the natural density, it is easy to convert it into the density weighted by $\log p/p$ in our situation by partial summation. We omit the details here. For generic $r$ sufficiently large (depending on $P$), $P(x)-r$ is irreducible and the Galois group of the splitting field of $P(x)-r$ is the symmetric group $S_3$, and the subgroup $H$ in the statement above has order $2$. The union of the three conjugates of $H$ has size $4$, and thus the inner sum is $\sim 2/3$, as claimed.

\begin{appendix}

\section{Sieving arithmetic progressions in the large sieve regime}

In this appendix we give details on proving Theorem \ref{thm:sieve-ap-large} by adapting the arguments in Section 5 of \cite{GH14}. We recall the statement:

\begin{nonumtheorem}
Let $A\subset [N]$ be a subset. Let $k$ be a positive integer and $\epsilon>0$ be real. Let $c>0$ be sufficiently small depending on $k$ and $\epsilon$. Let $\mathcal{P}$ be a subset of the primes up to $N^{1/2k}$ satisfying
\[ \sum_{\substack{p\notin\mathcal{P}\\ p\leq N^{1/2k}}}\frac{\log p}{p}<c\log N. \]
Suppose that for each $p\in\mathcal{P}$, there is an arithmetic progression $S_p\subset\Z/p\Z$ of length at most $(1-\epsilon)p$, such that $A\pmod p\subset S_p$. Then $|A|\ll N^{1/2-c}$.
\end{nonumtheorem}

Obviously the statement is stronger when $k$ is larger. We may thus assume that $k\geq 2$ and write $Q=N^{1/2k}$. Assume that $|A|\geq N^{1/2-c}$. We will construct a sequence of sets $A_0=A \supset A_1\supset A_2\supset\cdots$ satisfying
\begin{enumerate}
\item $|A_i|\geq N^{1/2-c_i}$, where $c_i=(3k)^ic$;
\item for each $p\in\mathcal{P}$ there is a set $S_p^i\subset\Z/p\Z$ such that $A_i\pmod p\subset S_p^i$ and
\[ \sum_{p\in\mathcal{P}}\frac{\log p}{p}\cdot\frac{|S_p^i|}{p}<(1-\eta)^i(\log Q+O(1)), \]
where $\eta=(10k)^{-8}\epsilon^4$ (say).
\end{enumerate}

Let us first deduce the theorem assuming these sets are constructed. For some $i=O_{k,\epsilon}(1)$ we have $(1-\eta)^i<1/4k$. Hence
\[ \sum_{p\leq Q}\frac{\log p}{p}\cdot\frac{|A_i\pmod p|}{p}\leq\sum_{p\in\mathcal{P}}\frac{\log p}{p}\cdot\frac{|S_p^i|}{p}+\sum_{\substack{p\notin\mathcal{P}\\ p\leq Q}}\frac{\log p}{p} 
<\frac{1}{3k}\log Q. \]
Apply the larger sieve (Lemma \ref{lem:gallagher2}) to $A_i$ we get $|A_i|\ll Q\leq N^{1/4}$. This contradicts the lower bound $|A_i|\geq N^{1/2-c_i}$ if $c$ is small enough.

Now we construct the sets $A_i$. Suppose that we have chosen $A_i$ and $S_p^i$ ($p\in\mathcal{P}$) satisfying the above properties. We may assume that
\[ \frac{1}{4k}(\log Q+O(1))\leq (1-\eta)^{i+1}(\log Q+O(1))\leq \sum_{p\in\mathcal{P}}\frac{\log p}{p}\cdot\frac{|S_p^i|}{p}<(1-\eta)^i(\log Q+O(1)). \]
(If the lower bound above fails, then we may take $A_{i+1}=A_i$). By Lemma \ref{lem:lift}, we have
\[ E(A_i)\gg |A_i|^3 N^{-c_i}\left(\frac{|\mathcal{P}|}{Q}\right)^{2k-1}\gg |A_i|^3 N^{-3kc_i}. \]
By a standard additive combinatorial argument (Lemma 3.6 in \cite{GH14}), there is a subset $H\subset [-N,N]$ with $|H|\gg |A_i|N^{-3kc_{i}}$ such that $|A_i\cap (A_i+h)|\gg |A_i|N^{-3kc_{i}}$ for each $h\in H$. We will take $A_{i+1}$ to be $A_i\cap (A_i+h)$ for an appropriate $h\in H$. It remains to choose $h\in H$ such that property (2) holds with $A_{i+1}=A_i\cap (A_i+h)$. This will be achieved by the following lemma, which completes our construction of $A_{i+1}$ and the proof.


\begin{nonumlemma}
Let the notations be as above. There exists $h\in H$ such that
\[ \sum_{p\in\mathcal{P}}\frac{\log p}{p}\cdot\frac{|S_p^i\cap (S_p^i+h)|}{p}<(1-\eta)\sum_{p\in\mathcal{P}}\frac{\log p}{p}\cdot\frac{|S_p^i|}{p}. \]
\end{nonumlemma}

\begin{proof}
Assuming the contrary we get
\[ \sum_{h\in H}\sum_{p\in\mathcal{P}}\frac{\log p}{p}\cdot\frac{|S_p^i\cap (S_p^i+h)|}{p}\geq (1-\eta)|H|\sum_{p\in\mathcal{P}}\frac{\log p}{p}\cdot\frac{|S_p^i|}{p}. \]
Let $\mathcal{P}_{\text{bad}}$ be the set of those primes $p\in\mathcal{P}$ with
\[ \sum_{h\in H}\frac{|S_p^i\cap (S_p^i+h)|}{p}\geq (1-\eta^{1/2})|H|\cdot\frac{|S_p^i|}{p}. \]
Since
\[ \sum_{h\in H}\sum_{p\in\mathcal{P}}\frac{\log p}{p}\cdot\frac{|S_p^i\cap (S_p^i+h)|}{p}\leq |H|\sum_{p\in\mathcal{P}_{\text{bad}}}\frac{\log p}{p}\cdot\frac{|S_p^i|}{p}+(1-\eta^{1/2})|H|\sum_{p\in\mathcal{P}\setminus\mathcal{P}_{\text{bad}}} \frac{\log p}{p}\cdot\frac{|S_p^i|}{p}, \]
we deduce that $\mathcal{P}_{\text{bad}}$ is large:
\[ \sum_{p\in \mathcal{P}_{\text{bad}}}\frac{\log p}{p}\cdot\frac{|S_p^i|}{p}\geq (1-\eta^{1/2})\sum_{p\in \mathcal{P}}\frac{\log p}{p}\cdot\frac{|S_p^i|}{p}\geq\frac{1}{5k}\log Q=\frac{1}{10k^2}\log N. \]
For $p\in\mathcal{P}_{\text{bad}}$, let $H_p\subset H$ be the set of those $h\in H$ with $|S_p^i\cap (S_p^i+h)|\geq (1-\eta^{1/4})|S_p^i|$. Since
\[ \sum_{h\in H}\frac{|S_p^i\cap (S_p^i+h)|}{p}\leq \frac{|S_p^i|}{p}|H_p|+(1-\eta^{1/4})\frac{|S_p^i|}{p} (|H|-|H_p|),  \]
we deduce that $|H_p|\geq (1-\eta^{1/4})|H|$. By Lemma 3.8 in \cite{GH14} (a consequence of Pollard's theorem) applied to $S_p^i$, we have $|H_p\pmod p|\leq 4\eta^{1/4} |S_p^i|+1\leq p/20k^2$. Apply the larger sieve as in Theorem 2.3 of \cite{GH14} to the set $H$ with those primes $p\in\mathcal{P}_{\text{bad}}$, we conclude that
\[ |H|\ll\frac{Q}{(1-\eta^{1/4})^2(20k^2)\sum_{p\in\mathcal{P}_{\text{bad}}}\frac{\log p}{p}-\log N}\ll Q\leq N^{1/4}. \]
This contradicts the lower bound on the size of $H$.
\end{proof}


\end{appendix}

\bibliographystyle{plain}
\bibliography{sieve}{}

\begin{thebibliography}{10}

\bibitem{BB96}
D.~Berend and Y.~Bilu.
\newblock Polynomials with roots modulo every integer.
\newblock {\em Proc. Amer. Math. Soc.}, 124(6):1663--1671, 1996.

\bibitem{Bom87}
E.~Bombieri.
\newblock Le grand crible dans la th\'eorie analytique des nombres.
\newblock {\em Ast\'erisque}, (18):103, 1987.

\bibitem{CE04}
E.~S. Croot, III and C.~Elsholtz.
\newblock On variants of the larger sieve.
\newblock {\em Acta Math. Hungar.}, 103(3):243--254, 2004.

\bibitem{Els01}
C.~Elsholtz.
\newblock The inverse {G}oldbach problem.
\newblock {\em Mathematika}, 48(1-2):151--158 (2003), 2001.

\bibitem{Els06}
C.~Elsholtz.
\newblock Additive decomposability of multiplicatively defined sets.
\newblock {\em Funct. Approx. Comment. Math.}, 35:61--77, 2006.

\bibitem{Els09}
C.~Elsholtz.
\newblock A survey on additive and multiplicative decompositions of sumsets and
  of shifted sets.
\newblock In {\em Combinatorial number theory and additive group theory}, Adv.
  Courses Math. CRM Barcelona, pages 213--231. Birkh\"auser Verlag, Basel,
  2009.

\bibitem{EH13}
C.~Elsholtz and A.~Harper.
\newblock Additive decompositions of sets with restricted prime factors.
\newblock {\em Arxiv preprint arXiv:1309.0593}, 2013.

\bibitem{Gal71}
P.~X. Gallagher.
\newblock A larger sieve.
\newblock {\em Acta Arith.}, 18:77--81, 1971.

\bibitem{Gom88}
J.~Gomez-Calderon.
\newblock A note on polynomials with minimal value set over finite fields.
\newblock {\em Mathematika}, 35(1):144--148, 1988.

\bibitem{GR05}
B.~Green and I.~Z. Ruzsa.
\newblock Sum-free sets in abelian groups.
\newblock {\em Israel J. Math.}, 147:157--188, 2005.

\bibitem{GH14}
B.~J Green and A.~J Harper.
\newblock Inverse questions for the large sieve.
\newblock {\em GAFA}.
\newblock To appear.

\bibitem{Gry13}
D.~J. Grynkiewicz.
\newblock {\em Structural additive theory}, volume~30 of {\em Developments in
  Mathematics}.
\newblock Springer, Cham, 2013.

\bibitem{HV09}
H.~A. Helfgott and A.~Venkatesh.
\newblock How small must ill-distributed sets be?
\newblock In {\em Analytic number theory}, pages 224--234. Cambridge Univ.
  Press, Cambridge, 2009.

\bibitem{Mon78}
H.~L. Montgomery.
\newblock The analytic principle of the large sieve.
\newblock {\em Bull. Amer. Math. Soc.}, 84(4):547--567, 1978.

\bibitem{PSS88}
C.~Pomerance, A.~S{\'a}rk{\"o}zy, and C.~L. Stewart.
\newblock On divisors of sums of integers. {III}.
\newblock {\em Pacific J. Math.}, 133(2):363--379, 1988.

\bibitem{TV06}
T.~Tao and V.~Vu.
\newblock {\em Additive combinatorics}, volume 105 of {\em Cambridge Studies in
  Advanced Mathematics}.
\newblock Cambridge University Press, Cambridge, 2006.

\bibitem{Vau73}
R.~C. Vaughan.
\newblock Some applications of {M}ontgomery's sieve.
\newblock {\em J. Number Theory}, 5:64--79, 1973.

\bibitem{Wal12}
M.~N. Walsh.
\newblock The inverse sieve problem in high dimensions.
\newblock {\em Duke Math. J.}, 161(10):2001--2022, 2012.

\end{thebibliography}

\end{document}